\documentclass[11pt,a4paper]{article}

\usepackage{amsfonts,amssymb,amsthm,amsmath}
\usepackage{graphicx,pst-all}
\usepackage{amsthm}
\usepackage{float}
\usepackage{layout}

\numberwithin{equation}{section}

\newtheorem{theorem}{Theorem}[section]
\newtheorem{lemma}[theorem]{Lemma}

\newtheorem{corollary}[theorem]{Corollary}
\newtheorem{proposition}[theorem]{Proposition}
\newtheorem{remark}[theorem]{Remark}

\def\eI{[0,1]}

\newcommand{\Dom}{{\it{Dom}}}

\newcommand{\R}{{\mathbb{R}}}

\newcommand{\G}{{\mathcal{G}}}
\newcommand{\E}{{\mathcal{E}}}
\newcommand{\C}{{\mathcal{C}}}
\newcommand{\Pe}{{\mathcal{P}}}

\DeclareSymbolFont{AMSb}{U}{msb}{m}{n}
\newcommand{\N}{{\mathbb{N}}}

\newcommand{\Q}{{\mathbb{Q}}}
\newcommand{\Pp}{{\mathbb{P}}}

\newcommand{\D}{{\mathbb{D}}}

\newcommand{\EE}{{\mathbb{E}}}

\newcommand{\leb}{{\mbox{Leb}}}
\newcommand{\Cyl}{{\mathfrak{C}}}

\newcommand{\Zyl}{{\mathfrak{Z}}}
\def\smint{{\mbox{$\int$}}}

\newcommand{\I}{{\mathcal I}}
\newcommand{\J}{{\mathcal J}}
\newcommand{\HH}{{\mathcal H}}

\begin{document}

\title{A Monotone Approximation to the Wasserstein Diffusion}

\author{Karl-Theodor Sturm}

\date{}
\maketitle

\begin{section}{Introduction and Statement of the Main Results}

The Wasserstein space $\Pe(M)$ on an Euclidean or Riemannian space $M$ -- i.e. the space of probability measures on $M$ equipped with the $L^2$-Wasserstein distance $d_W$ -- offers a rich geometric structure. This allows to develop a far reaching {\em first order calculus},  with striking applications for instance to the reformulation of conservative PDEs on $M$ as gradient flows of suitable functionals on $\Pe(M)$, see e.g. \cite{Otto}, \cite{Villani}, \cite{AGS}.
A {\em second order calculus} was developed in \cite{renstu} in the particular case of a one-dimensional state space, say $M=[0,1]$, based on the construction of a canonical Dirichlet form
\begin{equation}\label{Wasserstein-P}
\EE_\Pe(u,v)=\int_{\Pe} \langle D u(\mu), D v(\mu)\rangle_{L^2(\mu)}^2 d\Pp^\beta(\mu)
\end{equation}
with domain $\D_\Pe\subset L^2(\Pe,\Pp^\beta)$.
Here $Du$ denotes the {\em Wasserstein gradient} and $\Pp^\beta$ a suitable measure ({\em "entropic measure"}). Among others, this leads to a canonical second order differential operator and to a canonical continuous Markov process $(\mu_t)_{t\ge0}$, called {\em Wasserstein diffusion}.

The goal of this paper is to derive approximations of these objects -- Dirichlet form, semigroup, continuous Markov process -- on the infinite dimensional space $\Pe:=\Pe([0,1])$ in terms of appropriate objects on finite dimensional spaces. In particular, we will approximate the Wasserstein diffusion in terms of interacting systems of Brownian motions.

\bigskip\bigskip

For each $k\in\N$ we consider the strongly local, regular Dirichlet form $(\mathcal E_k,\mathcal D_k)$ on $L^2(\R^k,\rho^\beta_k\,dx)$ defined on its core
$\mathcal C^1(\mathbb R^k)$ by

\begin{equation}\mathcal{E}_k(U,V)=k\int_{{\mathbb R}^k}\nabla U(x)\cdot\nabla V(x)\, \rho^\beta_k(x)\,dx.\end{equation}

The density
\begin{eqnarray*}
\lefteqn{\rho_\beta^k(x_1,\ldots,x_k)}\\
&=&\frac{\Gamma(\beta)e^\beta\beta^k}{[k\Gamma(\beta/k)]^k}
\int_{x_{k-1}}^{x_k}\ldots\int_{x_1}^{x_2}\prod_{i=1}^k
\left[
 \int_0^{\frac{x_i-y_{i-1}}{y_i-y_{i-1}}}\left(\frac{x_i-y_{i-1}}{y_i-y_{i-1}}-z_i\right)^{\beta/k-1}\cdot z_i^{- z_i\beta/k} \cdot (1-z_i)^{-(1-z_i)\beta/k} \cdot\right.\\
 &&\qquad\qquad
 \left.\cdot\left(y_i-y_{i-1}\right)^{\beta/k-2}\cdot\left(\cos(\pi z_i\beta/k)-\frac1\pi\sin(\pi z_i\beta/k)\cdot\log\frac{z_i}{1-z_i}\right)\, dz_i
\right]dy_1\ldots dy_{k-1}
\end{eqnarray*}
(where $y_0:=0, y_k:=1$)
is continuous, positive and bounded from above by
$$C\cdot
\left[x_1(1-x_k)\right]^{\beta/(2k)-1}\cdot\prod_{i=2}^{k}\left(x_i-x_{i-1}\right)^{\beta/k-1}$$
 on the simplex
$\Sigma_k:=\left\{(x_1,\ldots,x_k): 0<x_1<\ldots<x_k<1\right\}\subset\R^k$
and vanishes on $\R^k\setminus \Sigma_k$.

The strong Markov process $(X^k_t)_{t\ge0}=\left(X^{k,1}_t,\ldots,X^{k,k}_t\right)_{t\ge0}$ associated with the Dirichlet form $(\mathcal E_k,\mathcal D_k)$ is continuous, reversible and recurrent. At least on those stochastic intervals for which $X^k_t(\omega)\in\Sigma_k$ it can be characterized as the solution to an interacting system of stochastic differential equations
\begin{equation}
dX_t^{k,i}=k\frac{\partial \log\rho_k^\beta}{\partial x_i}\left(X_t^k \right)dt+\sqrt{2k}\,dW_t^i,\quad i=1,\ldots,k
\end{equation}
for some $k$-dimensional Brownian motion $(W_t)_{t\ge 0}$.

\bigskip\bigskip

In many respects, an alternative representation for (\ref{Wasserstein-P}) is be more convenient.
The map
$\chi: \  g\mapsto
g_*\leb|_{[0,1]}$  establishes an isometry between the set $\G$ of right continuous increasing functions
$g:[0,1)\to[0,1]$ and
$\Pe$. Here $\G$ will be regarded as a convex subset of the Hilbert space $L^2([0,1],\leb)$.
The image of the form (\ref{Wasserstein-P}) under the map $\chi^{-1}:\Pe\to\G$ is given by the form $(\EE,\D)$ on $L^2(\G,\Q^\beta)$ with
\begin{equation}\label{Wasserstein-G}
\EE(u,v)=\int_{\G} \langle{\mathbf D} u(g),  {\mathbf D} v(g)\rangle \, d\Q^\beta(g)
\end{equation}
where $\mathbf D u$ denotes the Frechet derivative for "smooth" functions $u:\G\to\R$ and
$\Q^\beta$ is the well-known Dirichlet-Ferguson process with parameter measure $\beta\cdot\leb|_{[0,1]}$.

\begin{theorem}
{\bf (i)} \
For each $k\in\N$ the Dirichlet form $(\mathcal E_k,\mathcal D_k)$ on $L^2(\R^k,\rho^\beta_k\,dx)$ is isomorphic to a restriction
$(\EE,\D_k)$  of  the Dirichlet form $(\EE,\D)$ on $L^2( L^2([0,1],\leb), \Q^\beta)$. The isomorphism is induced by the embedding $$\iota: x\mapsto \sum_{i=1}^k x_i\cdot 1_{[\frac{i-1}k,\frac i k)}$$  of $\mathbb R^k$ into $L^2([0,1],\leb)$ (and of $\Sigma_k$ into $\G$).

{\bf (ii)}\
The semigroup $\mathbb T_t^k$ associated with $(\EE,\D_k)$ is given explicitly in terms of the semigroup $T_t^k$ of the Dirichlet form $(\mathcal E_k,\mathcal D_k)$. If $g=\iota(x)$ for some $x\in\R^k$  then
$$\mathbb T_t^ku(g)=T_t^kU(x)$$
with $U:=u\circ \iota.$

{\bf (iii)}\
The strong Markov process $(g_t^k)_{t\ge0}$ on $\G$ associated with $(\EE,\D_k)$ is given by
$$g_t^k=\sum_{i=1}^k X_t^{k,i}\cdot 1_{[\frac{i-1}k,\frac i k)}$$
if $g_0=\iota(x_0)$  and if $(X_t^k)_{t\ge0}$ denotes the Markov process on $\R^k$ associated with  $(\mathcal E_k,\mathcal D_k)$
with initial condition $X_0^k=x_0$.

{\bf (iv)}\
A strong Markov process $(\mu_t^k)_{t\ge0}$ on $\Pe$ (not necessarily normal) is defined by
$$\mu_t^k(\omega)=\left(g_t^k(\omega)\right)_*\leb|_{[0,1]}=\frac1k\sum_{i=1}^k\delta_{X_t^{k,i}(\omega)}$$
that is, as the empirical distribution of the process $(X^k_t)_{t\ge0}$.
It is continuous, recurrent and reversible with invariant distribution
$\Pp_k^\beta=(\iota_\Pe)_*m_k^\beta$ obtained as push forward of the measure $m_k^\beta(dx)=\rho_k^\beta(x)dx$ under the embedding
$$\iota_\Pe: \overline\Sigma_k\to\Pe,\ x\mapsto \frac1k\sum_{i=1}^k x_i.$$
\end{theorem}

\begin{theorem} {\bf (i)}\
The domains $\D_{2^k}$ are increasing in $k\in\N$ with $\D=\overline{\cup_k \D^{2^k}}$.
Therefore,
$$(\EE,\D_{2^k})\to (\EE,\D)\quad\mbox{ in the sense of Mosco}$$ and, hence,
for the associated semigroups and resolvents
\begin{equation}\label{conv-t}\mathbb T^{2^k}_t\to \mathbb T_t,\quad \mathbb G^{2^k}_\alpha \to \mathbb G_\alpha\quad\mbox{strongly in $L^2(\G,\Q^\beta)$ as }k\to\infty.\end{equation}

{\bf (ii)}\
For the associated Markov processes  on $\Pe$ starting from the respective invariant distributions we obtain convergence
\begin{equation}\label{conv-mu}(\mu_t^{2^k})_{t\ge0}\to (\mu_t)_{t\ge0}\quad\mbox{as }k\to\infty\end{equation}
in distribution weakly on $\C(\R_+,\Pe)$.
\end{theorem}

A closely related approximation result has been presented by Sebastian Andres and Max-K. von Renesse \cite{AndRes}. Their finite dimensional objects are more explicit; the convergence issues in their approximation, however, are quite delicate.

\end{section}
\begin{section}{Dirichlet-Ferguson Process, Entropic Measure and Wasserstein Diffusion}
\begin{subsection}{The Dirichlet-Ferguson Process}

Let $\G$ denote the space of all right continuous nondecreasing
maps $g: [0,1]\,\to\eI$ with $g(1)=1$. We will regard $\G$ as a convex subset of the Hilbert space $L^2([0,1],\leb)$.
The scalar product in $L^2([0,1],\leb)$ will always be denoted by $\langle.,.\rangle$.

\begin{proposition}
For each real number $\beta>0$ there exists a unique probability
measure $\Q^\beta$ on $\G$, called {\em Dirichlet-Ferguson process}, with
the property that for each $k\in\N$ and each  family
$0=t_0<t_1<t_2<\ldots<t_{k-1}<t_{k}=1$
\begin{equation}\label{Dir-I}
\Q^\beta\left(g_{t_1}\in dx_1,\ldots,g_{t_{k-1}}\in dx_{k-1}\right)=
\frac{\Gamma(\beta)}{\prod_{i=1}^k\Gamma(\beta\cdot (t_{i}-t_{i-1}))}
\prod_{i=1}^k (x_{i}-x_{i-1})^{\beta\cdot (t_{i}-t_{i-1})-1} dx_1\ldots dx_{k-1}.
\end{equation}
\end{proposition}

The Dirichlet-Ferguson process can be identified with the normalized distribution of the standard Gamma process $(\gamma_t)_{t\geq 0}$:
{\it For each $\beta >0$, the law of the process   $(\frac{\gamma_{t\cdot
\beta}}{\gamma_\beta})_{t \in \eI}$  is the Dirichlet-Ferguson process
$\Q^\beta$.}

Recall that a right continuous, real valued Markov process
$(\gamma_t)_{t\geq 0}$ starting in zero is called standard Gamma
process  if its increments $\gamma_t - \gamma_s$  are independent and
distributed for $0\leq s< t$ according to    $G_{t-s}(dx)
=\frac 1
{\Gamma(t-s)} 1_{[0,\infty)}(x)x^{t-s -1} e^ {-x}dx$.

\medskip

In \cite{renstu} as well as in \cite{RYZ} a change of variable formula (under composition) has been derived for the Dirichlet-Ferguson process.

\end{subsection}

\begin{subsection}{The Dirichlet Form on $\G$}

Let $\Cyl^1(\G)$
denote the set of all ('cylinder') functions $u: \G\to\R$ which can be written as
 $u(g)=U\left(\langle g,\psi_1\rangle,\ldots,\langle g,\psi_n\rangle\right)$ with $n\in\mathbb{N}$, $U\in \C^1(\R^n,\R)$ and
$\psi_1,\ldots,\psi_n\in L^2([0,1],\leb)$. For $u$ of this form
the gradient
$${\mathbf D} u(g)=\sum_{i=1}^n \partial_i U\left(\langle g,\psi_1\rangle,\ldots,\langle g,\psi_n\rangle\right)\cdot \psi_i(.)$$
exists in $L^2(\eI,\leb)$
and
$$\|{\mathbf D} u(g)\|^2=
\int_0^1\left|\sum_{i=1}^n \partial_i U\left(\langle g,\psi_1\rangle,\ldots,\langle g,\psi_n\rangle\right)\cdot \psi_i(s)\right|^2ds.$$

For $u,v\in  \Cyl^1(\G)$ we define the  Dirichlet integral
\begin{equation}\label{Wasserstein-dir-int}
\EE(u,v)=\int_{\G} \langle{\mathbf D} u(g),  {\mathbf D} v(g)\rangle \, d\Q^\beta(g).
\end{equation}

\begin{theorem}[\cite{renstu} Thm. 7.5, 7.8, \cite{DoeSta}]

{\bf (i)} \ $(\EE,\Cyl^1(\G))$ is closable. Its closure $(\EE,
\D)$ is a regular, strongly local, recurrent Dirichlet form on $L^2(\G,\Q^\beta)$.

{\bf(ii)} \ The associated Markov process $(g_t)_{t\ge0}$ on $\G$ is continuous, reversible and recurrent.

{\bf(iii)} \
The Dirichlet form $(\EE, \D)$ satisfies a logarithmic Sobolev inequality with constant $\frac1\beta$.
\end{theorem}

\end{subsection}

\begin{subsection}{The Dirichlet Form on the Wasserstein Space}
Let
$\Pe=\Pe([0,1])$  denote the space of
probability measures on
 the unit interval $[0,1]$.
The map
$\chi: \ \G\to\Pe, \ g\mapsto
g_*\leb|_{[0,1]}$  establishes a bijection between $\G$ and
$\Pe$. The inverse map $\chi^{-1}: \ \Pe\to\G, \
\mu\mapsto g_\mu$ assigns to each probability measure  $\mu\in\Pe$
its  inverse distribution function defined by
$g_\mu(t):= \inf \{ s\in\eI: \ \mu[0,s]  > t\}$
with $\inf\emptyset:=1$. The $L^2$-Wasserstein distance on $\Pe$ is characterized by
$d_W(\mu,\nu)=\|g_\mu-g_\nu\|_{L^2}$ for all $\mu,\nu\in\Pe$.

\smallskip

The {\em entropic measure} $\mathbb{P}^\beta$ on
$\Pe=\Pe(\eI)$ is defined as the push forward of the Dirichlet
process $\Q^\beta$ on $\G$ under the map $\chi$.

\begin{corollary}[\cite{renstu} Thm. 7.17]\label{wasserstein-dir-form}
The image of the
Dirichlet form defined above under the map $\chi$ is the regular, strongly
local, strongly local, recurrent  Dirichlet form $\EE_\Pe$ on $L^2(\Pe,\Pp^\beta)$, defined on its core $\Zyl^1(\Pe)$ by
\begin{equation}\label{p-wasserstein-form}
\EE_\Pe(u,v)=\int_{\Pe} \langle D u(\mu), D v(\mu)\rangle_{L^2(\mu)}^2 d\Pp^\beta(\mu).
\end{equation}
The associated Markov process $(\mu_t)_{t\ge0}$ on $\Pe$, called Wasserstein diffusion, is given by
$$\mu^{(\omega)}_t=(g_t^{(\omega)})_*\leb|_{[0,1]}.$$
\end{corollary}

Here $\Zyl^1(\Pe)$ denotes the set of all  functions $u: \Pe\to\R$ which can be written as
$u(\mu)=U\left(\int_0^1 \Psi_1d\mu,\ldots,\int_0^1\Psi_nd\mu\right)$
with some $n\in\N$, some $U\in\C^1(\R^n)$ and some $\Psi_1,\ldots,\Psi_n\in\C^1(\eI)$.
For $u$  as above we define its 'Wasserstein gradient' $D u(\mu)\in L^2(\eI,\mu)$ by
$$D u (\mu)=\sum_{i=1}^n \partial_i U(\smint \Psi_1 d\mu,\ldots,\smint \Psi_nd\mu)\cdot \Psi_i'(.)$$
with norm
$$\|D u(\mu)\|_{L^2(\mu)}=\left[
\int_0^1\left|\sum_{i=1}^n \partial_i U(\smint \Psi_1 d\mu,\ldots,\smint \Psi_nd\mu)\cdot \Psi_i'\right|^2d\mu\right]^{1/2}.$$
Recall that the tangent space at a given point $\mu\in\Pe$ can be identified with $L^2(\eI,\mu)$.

\smallskip

The analogue to (\ref{p-wasserstein-form}) on multidimensional spaces has been constructed in \cite{Sturm}.
\end{subsection}
\end{section}

\begin{section}{The Distribution of Random Means}

Let $m_1^\beta=\zeta_*\Pp^\beta$ denote the distribution of the random variable $\zeta: \mu\mapsto\int_0^1x\,d\mu(x)$ which assigns to each probability measure $\mu\in\Pe$ its mean value ({\em random means} of the {\em random probability measure} $\Pp^\beta$). Actually, $m_1^\beta$ coincides with the distribution of the random means of the random probability measure $\Q^\beta$, that is, $m_1^\beta=\tilde\zeta_*\Q^\beta$ where
$\tilde\zeta: g\mapsto \int_0^1t\,dg(t)$ assigns to each function $g\in\G$ the mean value of the probability measure $dg$.

Indeed, integration by parts yields
$\int_0^1t\,dg(t)=\int_0^1(1-g(t))\,dt=\int_0^1(1-x)d\mu(x)$ for $\mu=g_*\leb$.
Due to the symmetry of the entropic measure under the transformation $x\mapsto 1-x$ the distribution of $\int_0^1(1-x)d\mu(x)$ coincides with $m_1^\beta$.

The law of the  random means of the Dirichlet-Ferguson process is a well studied quantity.
Let $\Theta_\beta$   be the distribution function of $m_1^\beta$.   For simplicity, we will restrict ourselves in this section to the case $\beta\in(0,1)$. The following result can be found e.g. in \cite{Regazzini-Guglielmi-Nunno}, Proposition 8 and Proposition 3.

\begin{lemma}
$\Theta_\beta$ admits the following representations
\begin{equation*}\Theta_\beta(x)=\frac12+\frac1\pi\int_0^\infty\exp\left(-\frac\beta2\int_0^1\log\left[1+t^2(x-y)^2\right]dy\right)\cdot
\sin\left(\beta\int_0^1\arctan\left[t(x-y)\right]dy\right)\frac{dt}t\end{equation*}
and
\begin{equation*}\Theta_\beta(x)=\frac{e^\beta}{\pi}\int_0^x(x-y)^{\beta-1}\cdot y^{-\beta y}\cdot (1-y)^{-\beta(1-y)}\cdot \sin(\pi\beta y)\,dy.\end{equation*}
\end{lemma}

\begin{proposition}\label{mu_1} The measure $m_1^\beta$ is absolutely continuous with density $\vartheta_\beta=(\Theta_\beta)'$ given by
\begin{equation}\label{vartheta}\vartheta_\beta(x)=\beta e^\beta\int_0^x(x-y)^{\beta-1}\cdot y^{-\beta y} \cdot (1-y)^{-\beta(1-y)} \cdot\left[\cos(\pi\beta y)-\frac1\pi\sin(\pi\beta y)\cdot\log\frac y{1-y}\right]\, dy.\end{equation}
\end{proposition}

\begin{figure}
\begin{center}
 \includegraphics[width=0.8\textwidth]{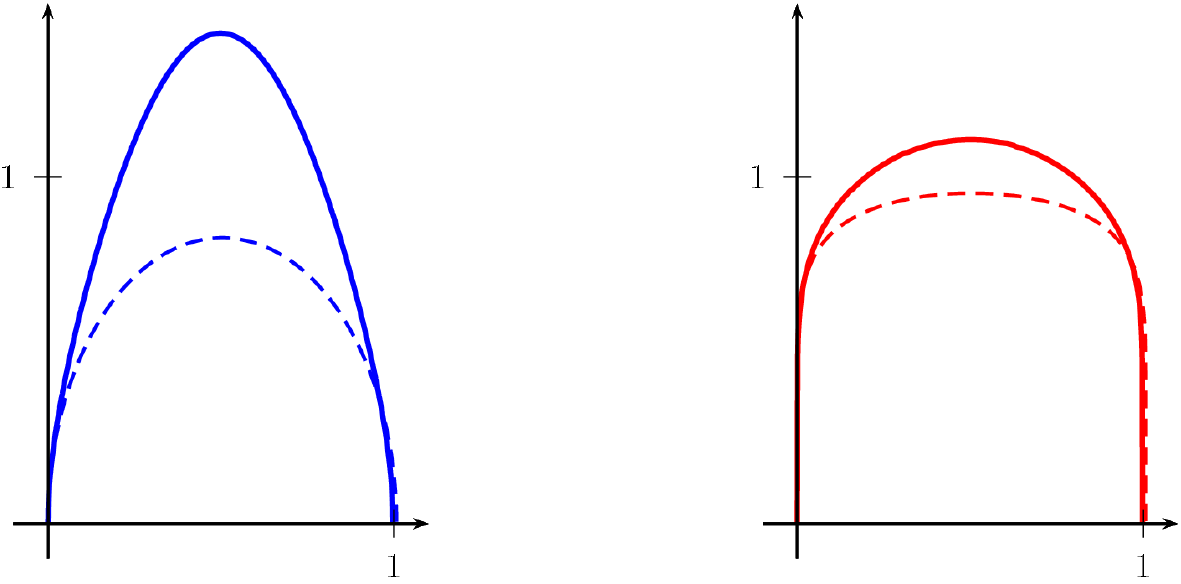}
\end{center}

\end{figure}



\begin{proof} The proof requires some care since  we are interested in the case $\beta<1$. Put
\begin{equation*}
\eta(y)=\frac{e^\beta}{\beta\pi}\cdot y^{-\beta y}\cdot (1-y)^{-\beta(1-y)}\cdot \sin(\pi\beta y)
\end{equation*}
in order to obtain
\begin{equation*}
\Theta_\beta(x)=\beta\int_0^x(x-y)^{\beta-1}\cdot \eta(y)\,dy=\beta\int_0^xy^{\beta-1}\cdot \eta(x-y)\,dy.
\end{equation*}
Differentiating the latter yields (since $\eta(x-y)\searrow0$ for $y\nearrow x$)
\begin{equation*}
\vartheta_\beta(x)=\beta\int_0^xy^{\beta-1}\cdot \eta'(x-y)\,dy=\beta\int_0^x(x-y)^{\beta-1}\cdot \eta'(y)\,dy.
\end{equation*}
Moreover, calculating $\eta'$ gives
\begin{equation*}\eta'(y)=
e^\beta\cdot y^{-\beta y} \cdot(1-y)^{-\beta(1-y)} \cdot\left[\cos(\pi\beta y)-\frac1\pi\sin(\pi\beta y)\cdot\log\frac y{1-y}\right].\end{equation*}
This proves the claim.
\end{proof}

\begin{proposition}
The density $\vartheta: [0,1]\to\mathbb R$ has the following properties
\begin{enumerate}
\item
$\vartheta$ is symmetric, i.e.  $\vartheta(x)=\vartheta(1-x)$;
\item $\vartheta$ is continuous on $[0,1]$ and $\mathcal C^\infty$ on $(0,1)$;
\item $\vartheta>0$ on $(0,1)$ and $\vartheta(0)=\vartheta(1)=0$;
\item $\vartheta(x)/\tilde\vartheta(x)\to1$ as $x\to0$ or $x\to1$ for
$\tilde\vartheta(x):=[e\cdot x(1-x)]^\beta$;
\item $\exists C\ge c>0$, e.g. $c=\cos(\pi\beta/2)$ and $C=4^\beta[1+\beta/e]$,  s.t. for all $x\in[0,1]$
\begin{equation}\label{tilde-theta}c \tilde\vartheta(x)\le\vartheta(x)\le C \tilde\vartheta(x).\end{equation}
\end{enumerate}

\end{proposition}

\begin{proof}
(i) is proven in \cite{Regazzini-Guglielmi-Nunno}, Proposition 6. It also follows immediately from formula (4.1).

(ii) The smoothness inside $(0,1)$ follows from the representation formula in the previous Proposition. Continuity at the boundary is a consequence of the estimates in (iv).

(iii) is a consequence of (v).

(iv) Using the notations from the proof of the previous Proposition and the fact that $\eta'(y)\to e^\beta$ as $y\to0$  we obtain
$$\frac{\vartheta(x)}{(e\cdot x)^\beta}=\frac{\beta}{(e\cdot x)^\beta}\int_0^x (x-y)^{\beta-1}\cdot \eta'(y)\,dy\quad\to\quad
\frac{\beta}{x^\beta}\int_0^x (x-y)^{\beta-1}\,dy=1$$
as $x\to0$. Combined with the symmetry (i) this proves the claim.

(v) A lower estimate of the form
$$\vartheta(x)\ge (e\cdot x)^\beta\cdot \cos(\pi\beta/2)$$
for $x\le1/2$ follows from the estimate
$\eta'(y)\ge e^\beta\cdot \cos(\pi\beta/2)$,  valid for all $y\le1/2$,

On the other hand, the estimate
$$\eta'(y)\le (2e)^\beta\cdot\left[\cos(\pi\beta y)-\frac1\pi\sin(\pi\beta y)\cdot\log\frac y{1-y}\right]
\le (2e)^\beta\cdot\left[1+\frac{\beta}{e}\right],$$
again valid for $y\le1/2$, implies
$$\vartheta(x)\le (2ex)^\beta\cdot\left[1+\frac{\beta}{e}\right]$$
for all $x\le1/2$. Due to the symmetry of $\vartheta$ this proves the claim.
\end{proof}

\begin{remark} For all $x\in(0,1)$
\begin{itemize}
\item
$\Theta_\beta(x)\to x$ and $\vartheta_\beta(x)\to 1$ as $\beta\to0$
\item
$\Theta_\beta(x)\to \frac12 \cdot 1_{\{\frac12\}}(x) +1_{(\frac12,1]}(x)$ as $\beta\to\infty$.
\end{itemize}
\end{remark}
\end{section}

\begin{section}{The Measure $m_k^\beta$ in the Multivariate Case}

>From a technical point of view, the main result of this paper is the identification of the distribution
 of the random vector
\begin{equation}\hat\J_k(g)=\left( \int_0^1\Phi_k^{(1)}dg,\ldots,\int_0^1\Phi_k^{(k)}dg\right)\end{equation}
under $\Q^\beta$ where
\begin{equation}\label{Phi}\Phi_k^{(i)}(t):=\left\{
\begin{array}{ll}1,& \mbox{ for }t\in[0,\frac{i-1}k]\\
i-kt,\quad & \mbox{ for }t\in[\frac{i-1}k,\frac ik]\\
0, &\mbox{ for }t\in[\frac ik,1].
\end{array}
\right.
\end{equation}

\begin{figure}
\begin{center}
  \includegraphics[width=0.7\textwidth]{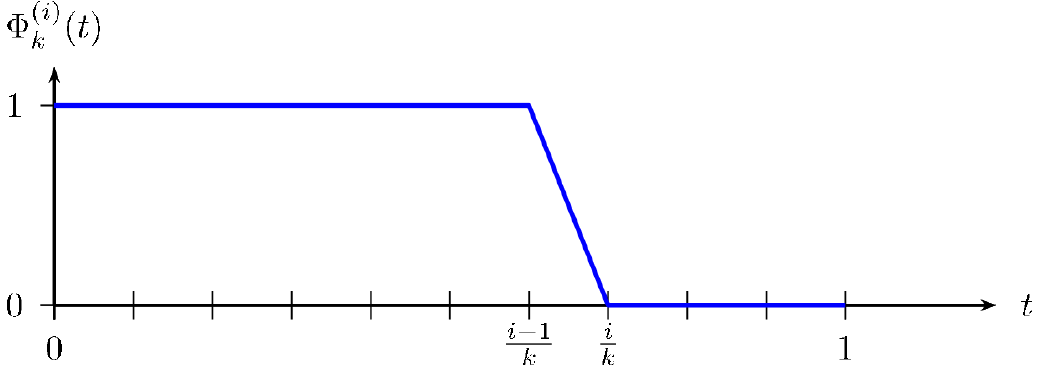}
\end{center}

\end{figure}



Note that integration by parts yields
$$\int_0^1\Phi_k^{(i)}(t)dg(t)=k\int_{\frac{i-1}k}^{\frac i k}g(t)dt$$
 for all $i=1,\ldots,k$ and all $g\in\G$.
Put $$m_k^\beta:=\left(\hat\J_k\right)_*\Q^\beta.$$

\begin{theorem} For any $\beta>0$ and $k\in\N$, $k\ge\beta$, the measure $m_k^\beta$ on $\mathbb R^k$ is absolutely continuous.
The density is strictly positive and continuous on the simplex
$$\Sigma_k:=\left\{(x_1,\ldots,x_k): 0<x_1<\ldots<x_k<1\right\}\subset\R^k$$
and vanishes on $\R^k\setminus \Sigma_k$.
For $x\in\Sigma_k$ it is given by
\begin{equation}\label{rho}
\rho_k^\beta(x_1,\ldots,x_k)=\frac{\Gamma(\beta)}{\Gamma(\beta/k)^k}
\int_{x_{k-1}}^{x_k}\ldots\int_{x_1}^{x_2}\prod_{i=1}^k
\left[\vartheta_{\beta/k}\left(\frac{x_i-y_{i-1}}{y_i-y_{i-1}}\right)\cdot\left(y_i-y_{i-1}\right)^{\beta/k-2}\right]dy_1\ldots dy_{k-1}
\end{equation}
(where $y_0:=0, y_k:=1$) with $\vartheta_\beta$ as defined in (\ref{vartheta}).
\end{theorem}

\begin{proof}  Let us start with the simple observation that
$$\int_0^1\Phi_k^{(i)}dg=g\left(\frac{i-1}k\right)+
\left[g\left(\frac{i}k\right)-g\left( \frac{i-1}k\right)\right]\cdot\int_0^1 (1-t)d\tilde g_i(t)$$
with $$\tilde g_i(t):=\frac{g\left(\frac{t+i-1}k\right)-g\left( \frac{i-1}k\right)}{g\left(\frac{i}k\right)-g\left( \frac{i-1}k\right)}.$$
Now the crucial fact is that, conditioned on $\left(g\left(\frac{1}k\right),\ldots,g\left( \frac{k-1}k\right)\right)$, the processes
$\left(\tilde g_i(t)\right)_{t\in[0,1]}$ for $i=1,\ldots,k$ are independent and distributed according to $\Q^{\beta/k}$. (This can be deduced from the explicit representation formula for the finite dimensional distributions
(\ref{Dir-I}), see also \cite{renstu}, Proposition 3.15).

\begin{figure}
\begin{center}
  \includegraphics[width=0.7\textwidth]{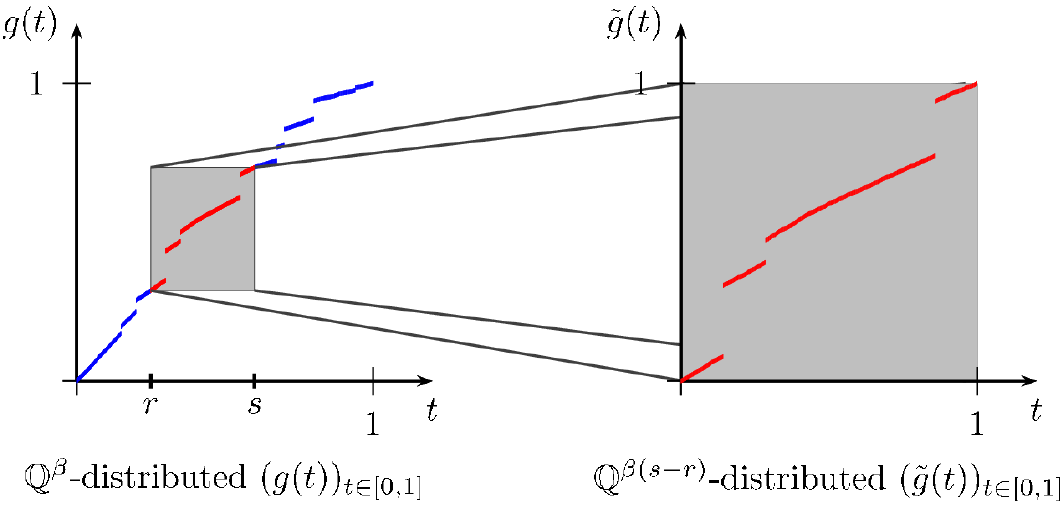}
\end{center}

\end{figure}



Moreover, according to Proposition \ref{mu_1} the distribution of $\int_0^1 (1-t)d\tilde g_i(t)$ for $\Q^{\beta/k}$-distributed $\left(\tilde g_i(t)\right)_{t\in[0,1]}$ is given by $dm_1^{\beta/k}(x)=\vartheta_{\beta/k}(x)\,dx$.

Finally, the distribution of the random vector $\left(g\left(\frac{1}k\right),\ldots,g\left( \frac{k-1}k\right)\right)$ is given explicitly by the Dirichlet distribution, see formula (\ref{Dir-I}).

\medskip

Putting these informations together we obtain for each bounded Borel function $U$ on $\R^k$
\begin{eqnarray*}
\lefteqn{\int_\G U\left(\left( \int_0^1 \Phi_k^{(i)}d g\right)_{i=1,\ldots,k}\right)\, d\Q^\beta}\\
&=&
\int_\G U\left(\left( g\left(\frac{i-1}k\right)+
\left[g\left(\frac{i}k\right)-g\left( \frac{i-1}k\right)\right]\cdot\int_0^1 (1-t)d\tilde g_i(t)\right)_{i=1,\ldots,k}\right)\, d\Q^\beta\\
&=&
\frac{\Gamma(\beta)}{\Gamma(\beta/k)^k} \int_{\Sigma_{k-1}}\left[\int_\G\ldots\int_\G U\left(\left( y_{i-1}+\left[y_i-y_{i-1}\right]\cdot\int_0^1 (1-t)d\tilde g_i(t)\right)_{i=1,\ldots,k}\right)\right.\\
&&\qquad\qquad\qquad \left. d\Q^{\beta/k}(\tilde g_1)\ldots  d\Q^{\beta/k}(\tilde g_k)\right]
\prod_{i=1}^k\left(y_i-y_{i-1}\right)^{\beta/k-1}\, dy_1\ldots dy_{k-1}\\
&=&
\frac{\Gamma(\beta)}{\Gamma(\beta/k)^k} \int_{\Sigma_{k-1}}\left[\int_0^1\ldots\int_0^1 U\left(\left(
y_{i-1}+\left[y_i-y_{i-1}\right]\cdot z_i \right)_{i=1,\ldots,k}\right)\right.\\
&&\qquad\qquad\qquad \left. \prod_{i=1}^k \vartheta_{\beta/k}(z_i)\, dz_1\ldots d z_k\right]
\prod_{i=1}^k\left(y_i-y_{i-1}\right)^{\beta/k-1}\, dy_1\ldots dy_{k-1}\\
&=&
\frac{\Gamma(\beta)}{\Gamma(\beta/k)^k} \int_{\Sigma_{k-1}}\left[\int_{y_{k-1}}^{y_k}\ldots\int_{y_0}^{y_1} U\left(\left(x_i \right)_{i=1,\ldots,k}\right)\right.\\
&&\qquad\qquad\qquad \left. \prod_{i=1}^k \left[\vartheta_{\beta/k}\left(\frac{x_i-y_{i-1}}{y_i-y_{i-1}}\right)
\cdot \left(y_i-y_{i-1}\right)^{\beta/k-2}
\right]
dx_1,\ldots dx_k \right]\, dy_1\ldots dy_{k-1}\\
&=&
\frac{\Gamma(\beta)}{\Gamma(\beta/k)^k} \int_{\Sigma_{k}}\left[\int_{x_{k-1}}^{x_k}\ldots\int_{x_1}^{x_2} U\left(\left(x_i \right)_{i=1,\ldots,k}\right)\right.\\
&&\qquad\qquad\qquad \left. \prod_{i=1}^k \left[\vartheta_{\beta/k}\left(\frac{x_i-y_{i-1}}{y_i-y_{i-1}}\right)
\cdot \left(y_i-y_{i-1}\right)^{\beta/k-2}
\right]
dy_1,\ldots dy_{k-1} \right]\, dx_1\ldots dx_{k}\\
&=&\int_{\Sigma_k} U\left(x_1,\ldots,x_k\right) \cdot \rho_k^\beta\left(x_1,\ldots,x_k\right) \, dx_1\ldots dx_k
\end{eqnarray*}
with $\rho_k^\beta$ as defined above (and always with $y_0:=0, y_k:=1$).

\medskip

The continuity and strict positivity of $\rho_k^\beta$ on $\Sigma_k$ follows from the explicit representation formula and from the fact that $\vartheta_{\beta/k}$ is smooth and $>0$ on $(0,1)$.
\end{proof}

\begin{remark}
The densities $\rho_k^\beta$ have the following hierarchical structure:
\begin{equation}
\rho_k(x_1,x_2,\ldots,x_k)=2^k\int_{\R^k}\rho_{2k}^\beta(x_1-\xi_1, x_1+\xi_1, \ldots, x_k-\xi_k,x_k+\xi_k)d\xi_1\ldots d\xi_k.
\end{equation}
This is of course a consequence of the fact that they are obtained via projection from the same measure $\Q^\beta$ and that
$$\Phi_k^{(i)}=\frac12\left( \Phi_{2k}^{(2i-1)}+\Phi_{2k}^{(2i)}\right)$$
for all $k\in\N$ and all $i=1,\ldots,k$.
Thus for all $U$ on $\R^k$
\begin{eqnarray*}
\int_{\R^k} U(x)\rho_k^\beta(x)dx&=&
\int_{\R^{2k}} U\left(\frac{y_1+y_2}2,\ldots, \frac{y_{2k-1}+y_{2k}}2\right)\rho_{2k}^\beta(y)dy\\
&=&
\int_{\R^k}U(x)\left[2^k\int_{\R^k}\rho_{2k}^\beta(x_1-\xi_1, x_1+\xi_1, \ldots, x_k-\xi_k,x_k+\xi_k)d\xi_1\ldots d\xi_k\right]dx.
\end{eqnarray*}

\end{remark}

\begin{proposition}
{\bf (i) \ } There exists a constant $C=C_{\beta,k}$ such that for all $x\in\Sigma_k$:
\begin{equation}\rho_k^\beta(x_1,\ldots,x_k)\le C\cdot
\left[x_1(1-x_k)\right]^{\beta/(2k)-1}\cdot\prod_{i=2}^{k}\left(x_i-x_{i-1}\right)^{\beta/k-1}.\end{equation}
{\bf (ii)} \  For all $l\in\{1,\ldots,k-1\}$ there exist continuous functions $\gamma_1>0$ on $\Sigma_l$ and $\gamma_2>0$ on $\Sigma_{k-l}$ such that
\begin{equation}
\rho_k^\beta(x)\ge \gamma_1(x_1,\ldots, x_l)\cdot\gamma_2(x_{l+1},\ldots,x_k)\cdot (x_{l+1}-x_l)^{2\beta/k-1}
\end{equation}
for all $x\in \Sigma_k$ with $|x_{l+1}-x_l|\le\frac14\min\{|x_{l}-x_{l-1}|,|x_{l+2}-x_{l+1}|\}$.
\end{proposition}
\begin{proof}
(i) Using the fact that $\vartheta_{\beta/k}\le C$ and the trivial estimate $(a+b)^{-p}\le 2^{-p}\cdot a^{-p/2}\cdot b^{-p/2}$ ($\forall a,b,p>0$) we obtain
\begin{eqnarray*}
\lefteqn{\rho_k^\beta(x_1,\ldots,x_k)}\\
&\le& C^k\cdot
\frac{\Gamma(\beta)}{\Gamma(\beta/k)^k}
\int_{x_{k-1}}^{x_k}\ldots\int_{x_1}^{x_2}\prod_{i=1}^k
\left(y_i-y_{i-1}\right)^{\beta/k-2}dy_1\ldots dy_{k-1}\\
&\le& C^k\cdot
\frac{\Gamma(\beta)}{\Gamma(\beta/k)^k}\cdot 2^{\beta-2k}
\int_{x_{k-1}}^{x_k}\ldots\int_{x_1}^{x_2}\prod_{i=1}^k
\left(y_i-x_{i}\right)^{\beta/(2k)-1}\cdot\left(x_i-y_{i-1}\right)^{\beta/(2k)-1}\, dy_1\ldots dy_{k-1}\\
&=&
C^k\cdot \frac{\Gamma(\beta)}{\Gamma(\beta/k)^k}
\left[\frac{\Gamma(\beta/(2k))^2}{\Gamma(\beta/k)}\right]^{k-1}\cdot 2^{\beta-2k}\cdot
\left[x_1(1-x_k)\right]^{\beta/(2k)-1}\cdot
\prod_{i=2}^{k}
\left(x_i-x_{i-1}\right)^{\beta/k-1}.
\end{eqnarray*}

\bigskip

(ii) We assume $k>2\beta$ and $2\le l\le k-2$. (The cases $l=1$ and $l=k-1$ require some modifications.)
Fix $x\in\Sigma_k$ as above and put $\delta:=|x_{l+1}-x_l|$. In the representation formula (4.3) for $\rho_k^\beta$, restrict the interval of integration for $dy_{l-1}$ from $[x_{l-1},x_l]$ to $[x_l-2\delta,x_l-\delta]$ and that for $dy_{l+1}$ from $[x_{l+1},x_{l+2}]$ to $[x_{l+1}+\delta,x_{l+1}+2\delta]$. Moreover, use the lower estimate (\ref{tilde-theta}) for the $\vartheta_{\beta/k}\left(\frac{x_i-y_{i-1}}{y_i-y_{i-1}}\right)$ for $i\in\{l,l+1\}$ to obtain the estimate
\begin{eqnarray*}
\lefteqn{\rho_k^\beta(x_1,\ldots,x_k)}\\
&\ge& C\cdot\int_{x_1}^{x_2}\ldots
\int_{x_{l-2}}^{x_{l-1}}
\int_{x_l-2\delta}^{x_l-\delta}
\int_{x_l}^{x_{l+1}}
\int_{x_{l+1}+\delta}^{x_{l+1}+2\delta}
\int_{x_{l+2}}^{x_{l+3}}\ldots
\int_{x_{k-1}}^{x_{k}}\\
&&\qquad\qquad
\prod_{i\in\{1,\ldots,l-1\}\cup\{l+2,\ldots,k\}}
\left[\vartheta_{\beta/k}\left(\frac{x_i-y_{i-1}}{y_i-y_{i-1}}\right)\cdot\left(y_i-y_{i-1}\right)^{\beta/k-2}\right]\cdot\\
&&\qquad\qquad\cdot (x_l-y_{l-1})^{\beta/k}\cdot(y_l-x_{l})^{\beta/k}\cdot(x_{l+1}-y_{l})^{\beta/k}\cdot(y_{l+1}-x_{l+1})^{\beta/k}\cdot\\
&&\qquad\qquad\cdot(y_l-y_{l-1})^{-\beta/k-2}\cdot(y_{l+1}-y_{l})^{-\beta/k-2}\,
dy_1\ldots dy_{k-1}.
\end{eqnarray*}
Here and in the rest of the proof $C$ always denotes a constant $>0$ changing from line to line. Now we use the lower estimates
\begin{eqnarray*}
(x_l-y_{l-1})^{\beta/k}\ge \delta^{\beta/k}&,&\qquad(y_{l+1}-x_{l+1})^{\beta/k}\ge \delta^{\beta/k},
\\
(y_l-y_{l-1})^{-\beta/k-2}\ge (3\delta)^{-\beta/k-2}&,&\quad(y_{l+1}-y_{l})^{-\beta/k-2}\ge (3\delta)^{-\beta/k-2},\\
(y_{l-1}-y_{l-2})^{\beta/k}\ge(x_l-y_{l-2})^{\beta/k}&,&\quad(y_{l+2}-y_{l+1})^{\beta/k}\ge(y_{l+2}-x_{l+1})^{\beta/k},\\
\vartheta_{\beta/k}\left(\frac{x_{l-1}-y_{l-2}}{y_{l-1}-y_{l-2}}\right)\ge\vartheta_{\beta/k}\left(\frac{x_{l-1}-y_{l-2}}{x_{l}-y_{l-2}}\right)&,&
\vartheta_{\beta/k}\left(\frac{x_{l+2}-y_{l+1}}{y_{l+2}-y_{l+1}}\right)\ge\vartheta_{\beta/k}\left(\frac{y_{l+2}-x_{l+2}}{y_{l+2}-x_{l+1}}\right)
\end{eqnarray*}
valid for all $y_{l-1},y_l,y_{+1}$ in the restricted domains of integration. Moreover, we put
\begin{eqnarray*}
\gamma_1(x_1,\ldots,x_l)&:=&\int_{x_1}^{x_2}\ldots
\int_{x_{l-2}}^{x_{l-1}}
\prod_{i\in\{1,\ldots,l-2\}}
\left[\vartheta_{\beta/k}\left(\frac{x_i-y_{i-1}}{y_i-y_{i-1}}\right)\cdot\left(y_i-y_{i-1}\right)^{\beta/k-2}\right]\cdot\\
&&\qquad\qquad\cdot\vartheta_{\beta/k}\left(\frac{x_{l-1}-y_{l-2}}{x_{l}-y_{l-2}}\right)\cdot(x_l-y_{l-2})^{\beta/k}\, dy_{l-2}\ldots dy_1
\end{eqnarray*}
and similarly
\begin{eqnarray*}
\gamma_2(x_{l+1},\ldots,x_k)&:=&\int_{x_{l+2}}^{x_{l+3}}\ldots
\int_{x_{k-1}}^{x_{k}}
\prod_{i\in\{l+3,\ldots,k\}}
\left[\vartheta_{\beta/k}\left(\frac{x_i-y_{i-1}}{y_i-y_{i-1}}\right)\cdot\left(y_i-y_{i-1}\right)^{\beta/k-2}\right]\cdot\\
&&\qquad\qquad\cdot
\vartheta_{\beta/k}\left(\frac{y_{l+2}-x_{l+2}}{y_{l+2}-x_{l+1}}\right)\cdot(y_{l+2}-x_{l+1})^{\beta/k}\, dy_{k-1}\ldots dy_{l+2}.
\end{eqnarray*}
Then we obtain
\begin{eqnarray*}
\lefteqn{\rho_k^\beta(x_1,\ldots,x_k)}\\
&\ge& C\cdot\gamma_1(x_1,\ldots,x_l)\cdot\gamma_2(x_{l+1},\ldots,x_k)\cdot\\
&&\cdot \delta^{-4}\cdot\int_{x_l-2\delta}^{x_l-\delta}
\int_{x_l}^{x_{l+1}}
\int_{x_{l+1}+\delta}^{x_{l+1}+2\delta}(y_l-x_l)^{\beta/k}\cdot(x_{l+1}-y_l)^{\beta/k}\,dy_{l-1}dy_ldy_{l+1}\\
&=&C\cdot\gamma_1(x_1,\ldots,x_l)\cdot\gamma_2(x_{l+1},\ldots,x_k)\cdot\delta^{2\beta/k-1}.
\end{eqnarray*}
This proves the claim.
\end{proof}

Remark: We do not know whether the exponent $2\beta/k-1$ in the previous lower estimate can  be improved to $\beta/k-1$. In the upper estimate, the exponent $\beta/k-1$ is certainly optimal.
\end{section}
\begin{section}{Projections, Isomorphisms, Approximations}

\begin{subsection}{Finite Dimensional Projections}

For each linear subspace $H\subset L^2(\eI,\leb)$ let $\Cyl^1_H(\G)$
denote the set of all  functions $u: \G\to\R$ which can be written as
 $u(g)=U\left(\langle g,\psi_1\rangle,\ldots,\langle g,\psi_n\rangle\right)$ with $n\in\mathbb{N}$, $U\in \C^1(\R^n,\R)$ and
$\psi_1,\ldots,\psi_n\in H$. Moreover, let $\D_H$ denote the closure of $\Cyl^1_H(\G)$ in $\D=\Dom(\EE)$
w.r.t. the norm $(\EE+\|.\|^2_{L^2(\Q^\beta)})^{1/2}$. Then  $(\EE, \D_H)$ is  a  -- not necessarily densely defined -- Dirichlet form on $L^2(\G,\Q^\beta)$.

\smallskip

More precisely,  let $\mathbb V_H$ denote the closure of $\D_H$ in $L^2(\G,\Q^\beta)$.
Then $(\EE, \D_H)$ is a closed quadratic form in $\mathbb V_H$.
As usual, there exist a strongly continuous semigroup $(\mathbb T^H_t)_{t\ge0}$ and a resolvent $(\mathbb G_\alpha^H)_{\alpha>0}$, both consisting of Markovian operators on  $\mathbb V_H$. Let $\pi_H: L^2(\G,\Q^\beta)\to\mathbb V_H$ be the orthogonal projection onto the closed linear subspace $\mathbb V_H$. Then a semigroup on $L^2(\G,\Q^\beta)$ -- not necessarily strongly continuous, however -- can be constructed by
\begin{equation}\label{Tthat}\hat{\mathbb T}^H_t:=\mathbb T^H_t\circ\hat{\pi}_H.
\end{equation}
The projection $\hat{\pi}_Hu$ of $u\in L^2(\G,\Q^\beta)$ can be characterized as the conditional expectation
$$\hat{\pi}_H u(g)=\int_\G u(\tilde g)\,\Q^\beta\left( d \tilde g \left| \left\{\langle \tilde g,\varphi\rangle= \langle  g,\varphi\rangle \mbox{ for all }\varphi\in H\right\}\right.\right)$$
of the
random variable $u: \tilde g\mapsto u(\tilde g)$ on $\G$ under the condition
$\left\{\langle \tilde g,\varphi\rangle= \langle  g,\varphi\rangle \mbox{ for all }\varphi\in H\right\}$.
\end{subsection}

\begin{subsection}{Monotone Convergence}

Let $(H(k))_{\in\mathbb N}$ be an increasing family of linear subspaces with  $L^2(\eI,\leb)=\overline{\bigcup_k H(k)}$ and define $\D_{H(k)}$ as above.
Then $\D_{H(k)}\nearrow$ with $\overline{\bigcup_k \D_{H(k)}}=\D$. In particular,
$$(\EE,\D_{H(k)}) \to (\EE,\D)\quad\mbox{in the sense of Mosco for }k\to\infty.$$
Hence, if $\hat{\mathbb T}^{H(k)}_t$ and $\hat{\mathbb G}^{H(k)}_\alpha$ denote the semigroup and resolvent operators on $L^2(\G,\Q^\beta)$ associated with $(\EE,\D_{H(k)})$ and if $\mathbb T_t$ and $\mathbb G_\alpha$ denote the corresponding  operators associated with $(\EE,\D)$ then
$$\hat{\mathbb T}^{H(k)}_t\to \mathbb T_t,\quad \hat{\mathbb G}^{H(k)}_\alpha \to \mathbb G_\alpha\quad\mbox{strongly in $L^2(\G,\Q^\beta)$ as }k\to\infty,$$
cf. \cite{ReeSim}.

\end{subsection}

\begin{subsection}{Isomorphisms I}

Let $H$ be finite dimensional with basis $\HH=\{\varphi^{(1)},\ldots,\varphi^{(k)}\}$ and consider the map
$$\hat\J_\HH:  \ L^2(\eI,\leb)\to {\mathbb R}^k,\quad
g\mapsto\left(\langle g,\varphi^{(1)}\rangle,\ldots,\langle g,\varphi^{(k)}\rangle\right).$$
Its restriction to $H$ --  denoted by $\J_\HH$ --  is a vector space isomorphism
with
$\J_\HH^{-1}:\ {\mathbb R}^k\to H,\quad x\mapsto \sum_{i,j=1}^k x_i a^{-1}_{ij}\varphi^{(j)}$
where $(a^{-1}_{ij})$ denotes the inverse of the matrix $(a_{ij})$ defined by
$a_{ij}=\langle \varphi^{(i)},\varphi^{(j)}\rangle$. This map induces an isomorphism between ${\mathcal C}^1(\mathbb{R}^k)$ and $\Cyl^1_H(\G)$:
$$U\in{\mathcal C}^1(\mathbb{R}^k)\stackrel{U=u\circ \J_\HH^{-1}}{\longleftrightarrow}u\in\Cyl^1_H(\G).$$
Let $m_\HH^\beta$ denote the distribution of the random vector $\left(\langle g,\varphi^{(1)}1\rangle,\ldots,\langle g,\varphi^{(k)}\rangle\right)$, that is,
$m_\HH^\beta:=(\hat\J_\HH)_*\Q^\beta$
and define a pre-Dirichlet form on $L^2(\mathbb{R}^k,m^\beta_\HH)=\left\{u\circ \J_\HH^{-1}:\ u\in {\mathbb V}_{H}\right\}$ by
\begin{equation}\label{pre-dir-k}\mathcal{E}_\HH(U,V):=\sum_{i,j=1}^k a_{ij}\int_{{\mathbb R}^k}\partial_i U(x)\partial_j V(x)\, dm^\beta_\HH(x)
\end{equation}
for $U,V\in {\mathcal C}^1(\mathbb{R}^k)$.
  This form  is closable -- since the closable form $(\EE,\Cyl^1_H(\G))$ is isomorphic to it -- with closure being a  strongly local Dirichlet form
on $L^2(\mathbb{R}^k,m_\HH^\beta)$ with domain
$${\mathcal D}_\HH=\left\{u\circ \J_\HH^{-1}:\ u\in {\mathbb D}_{H}\right\}$$
and with
$$\mathcal{E}_\HH(U,V)=\mathbb E(U\circ \hat\J_\HH,V\circ \hat\J_\HH)$$
for $U,V\in{\mathcal D}_\HH$, cf.
\cite{FukOshTak}.

\medskip

Let  $(T^\HH_t)_{t>0}$  denote the semigroup  associated with $(\mathcal E_\HH,\mathcal D_\HH)$. Then for all $u\in \mathbb V_H\subset L^2(\G,\Q^\beta)$
\begin{equation}\label{iso-1}
\mathbb T_t^Hu=\left(T_t^\HH U\right) \left(\hat\J_\HH\right)
\end{equation}
with $U\in L^2(\mathbb{R}^k,m_\HH^\beta)$
such that $u=U\circ \hat\J_\HH$.

\end{subsection}
\begin{subsection}{Standard Approximations}
For each $k\in\N$ let us from now on fix the linear subspace $H(k)\subset L^2([0,1],\leb)$ spanned by the orthogonal system
 $\HH(k)=\{\varphi_k^{(1)},\ldots,\varphi_k^{(k)}\}$ with
$$\varphi_k^{(i)}(t):=k\cdot 1_{(\frac{i-1}k\frac i k]}(t).$$
To simplify notation, write $m_k^\beta, \J_k, \E_k, T^k_t$ etc. instead of $m_{\HH(k)}^\beta, \J_{\HH(k)}, \E_{\HH(k)}, T^{\HH(k)}_t$, resp.

Note that in this case
\begin{equation*}\hat\J_k(g)=\left( k\int_{0}^{\frac 1 k}g(t)dt,\ldots,k\int_{\frac{k-1}k}^{1}g(t)dt \right)=\left( \int_0^1\Phi_k^{(1)}dg,\ldots,\int_0^1\Phi_k^{(k)}dg\right)\end{equation*}
with $\Phi_k^{(i)}$ as introduced in (\ref{Phi}).
Hence,the measure
$m_k^\beta:=(\hat\J_k)_*\Q^\beta$ on $\R^k$ coincides with the measure investigated in detail in the previous chapter. In particular,
$$dm_k^\beta(x)=\rho_k^\beta(x)\,dx$$
with $\rho_k^\beta$ given by formula (\ref{rho}). Recall that $\rho_k^\beta$ is continuous and $>0$ on the open simplex $\Sigma_k\subset\R^k$ and that it vanishes on $\R^k\setminus\Sigma_k$.

\medskip

The Dirichlet form $(\mathcal E_k,\mathcal D_k)$ on $L^2(\R^k,\rho_k^\beta)$ is given explicitly on its core
$\mathcal C^1(\R^k)$
by
\begin{equation}\label{Dir-endlich}\mathcal{E}_k(U,V)=k\int_{{\mathbb R}^k}\nabla U(x)\cdot\nabla V(x)\, dm^\beta_k(x)\end{equation}
with $\nabla U$ denoting the gradient of $U$ on $\R^k$. If we regard it as a Dirichlet form on $L^2(\overline{\Sigma_k},\rho_k^\beta)$ then it is {\em regular, strongly local and recurrent.} (Indeed, $\{u|_{\overline{\Sigma_k}}: u\in\mathcal C^1(\R^k)\}$ is dense in $\mathcal C(\overline{\Sigma_k})$ as well as in $\mathcal D_k$. Strong locality and recurrence is inherited from $(\EE,\D)$.)

The semigroup $(T_t^k)_{t\ge0}$ associated with $(\mathcal E_k,\mathcal D_k)$ can be represented
as
\begin{equation}\label{TtandXt}
T_t^ku(x)=\mathbb E_x\left[u\left(X_t^k\right)\right]
\end{equation}
(for all Borel functions $u\in L^2(\overline{\Sigma_k},\rho_k^\beta)$ and a.e. $x\in\overline{\Sigma_k}$)
in terms
of
 a strong Markov process
$$(X^k_t)_{t\ge0}=\left(X^{k,1}_t,\ldots,X^{k,k}_t\right)_{t\ge0}$$ with state space $\overline{\Sigma_k}$, defined on some probability space
$(\Omega,\mathcal F, \mathbf P_x)_{x\in\overline{\Sigma_k}}$ and canonically associated with $(\mathcal E_k,\mathcal D_k)$. This process is continuous, recurrent and reversible w.r.t. $m_k^\beta$. At least on those stochastic intervals for which $X^k_t(\omega)\in\Sigma_k$ it can be characterized as the solution to an interacting system of stochastic differential equations
\begin{equation}
dX_t^{k,i}=k\frac{\partial \log\rho_k^\beta}{\partial x_i}\left(X_t^k \right)dt+\sqrt{2k}\,dW_t^i,\quad i=1,\ldots,k
\end{equation}
for some $k$-dimensional Brownian motion $(W_t)_{t\ge 0}$.

\end{subsection}

\begin{subsection}{Isomorphisms II}

Let $\G_k:=\G\cap H(k)$ denote the subset of those $g\in\G$ which are constant on each of the intervals $[\frac{i-1}k,\frac i k)$ for $i=1,\ldots,k$. Then
$$\J^{-1}_k:\overline{\Sigma_k}\to\G_k, \ x\mapsto \sum_{i=1}^k x_i\cdot 1_{[\frac{i-1}k,\frac i k)}$$
is a bijection. It maps the strong Markov process $(X^k_t)_{t\ge0}$ on $\overline{\Sigma_k}$ onto a strong Markov process $(g_t^k)_{t\ge0}$ on $\G_k$ with
\begin{equation}\label{gt,Xt}
g_t^k(\omega):=\J^{-1}_k\left(X^k_t(\omega)\right)=\sum_{i=1}^k X_t^{k,i}(\omega)\cdot 1_{[\frac{i-1}k,\frac i k)}.\end{equation}

\bigskip

Now recall that the Hilbert space $\mathbb V_k:=\overline{\Cyl^1_k(\G)}^{L^2(\G,\Q^\beta)}$ coincides with $\left\{U\circ \J_k:\ U\in L^2(\mathbb{R}^k,m^\beta_k)\right\}$.
Hence, (\ref{iso-1}) together with (\ref{TtandXt}) and (\ref{gt,Xt}) imply
\begin{equation}\label{Tt,gt,Xt}
\mathbb T_t^ku(g)=\mathbf E_{g}\left[u\left(g_t^k\right)\right]=\mathbf E_{\J_k(g)}\left[u\left(\sum_{i=1}^k X_t^{k,i}\cdot 1_{[\frac{i-1}k,\frac i k)}\right)\right]
\end{equation}
for all Borel functions $u\in\mathbb V_k$ and a.e. $g\in\G$.
Finally, according to (\ref{Tthat})
\begin{equation}\label{Tt,gt,Xt2}
\hat{\mathbb T}_t^ku(g)=\mathbf E_{\J_k(g)}\left[u_k\left(\sum_{i=1}^k X_t^{k,i}\cdot 1_{[\frac{i-1}k,\frac i k)}\right)\right]
\end{equation}
for all Borel functions $u\in L^2(\G,\Q^\beta)$ and a.e. $g\in\G$ with
$u_k=\hat{\pi}_k u$ being the projection of $u$ onto $\mathbb V_k$ (or, in other words, the conditional expectation of $u$).

\bigskip

This process canonically extends to a -- not necessarily normal --  strong Markov process  $(g_t^k)_{t\ge0}$ on $\G$, projecting the initial data by means of the map
$$\pi_k:=\J_k^{-1}\circ \hat\J_k: \G\to\G_k, \ g\mapsto \frac1k\sum_{i=1}^k \langle g,\varphi_k^{(i)}\rangle \varphi_k^{(i)}.$$

\end{subsection}

\begin{subsection}{Isomorphisms III}

Let $\Pe_k$ denote the subset of $\mu\in\Pe$ which can be represented as $\mu=\frac1k\sum_{i=1}^k \delta_{x_i}$ for suitable $x_1,\ldots,x_k\in[0,1]$.
The maps
$\chi:\G_k\mapsto \Pe_k$
and $\I_k:=\J_k\circ \chi^{-1}: \Pe_k\to \overline \Sigma_k$ establish  canonical isomorphisms. The inverse of the latter
$$\I_k^{-1}:x\mapsto \frac1k\sum_{i=1}^k \delta_{x_i}$$
defines the canonical embedding of $\overline \Sigma_k$ into $\Pe$.
On the other hand, the map
$$\hat\I_k:=\hat\J_k\circ \chi^{-1}:\Pe\to \overline \Sigma_k$$
can be characterized as follows: Each $\mu\in\Pe$ can be represented uniquely as $\mu=\frac1k\sum_{i=1}^k\mu_i$ with probability measures $\mu_i$ supported on $[y_{i-1},y_i]$ for suitable $0\le y_1\le\ldots\le y_k\le1$. (Indeed, $y_i=\inf\{t\ge0: \mu([0,t])>\frac ik$ for each $i=1,\ldots,k$.)
Then
$$\hat\I_k(\mu)=\left(x_1,\ldots,x_k\right)$$ with $x_i=x_i^\mu=\int_0^1t\,d\mu_i(t)$
being the mean value of the probability measure $\mu_i$.

In particular, the projection
$\pi_k=\I_k^{-1}\circ\hat\I_k:\Pe\to\Pe_k$ is defined by
$$\mu\mapsto\frac1k\sum_{i=1}^k \delta_{x_i^\mu}.$$

\bigskip

Let $(\mu_t^k)_{t\ge0}$ be the image of the strong Markov process $(g_t^k)_{t\ge0}$ under the bijection $\chi: g\mapsto g_*\leb|_{[0,1]}$. Then
$$
\mu_t^k(\omega)=\frac1k\sum_{i=1}^k\delta_{X_t^{k,i}(\omega)}.$$
In other words, the strong Markov process $(\mu_t^k)_{t\ge0}$ on $\Pe_k$ is the empirical distribution of the strong Markov process $(X^k_t)_{t\ge0}$ on $\overline{\Sigma_k}$.

\bigskip

Finally, a probabilistic
representation -- similar to that for $\left(\hat{\mathbb T}_t^k\right)_{t\ge0}$ -- also holds true for the semigroup
$\left(\hat{\mathbb T}_{\Pe,t}^k\right)_{t\ge0}$ associated with the Dirichlet form $(\EE_\Pe,\D_\Pe)$
 on $L^2(\Pe,\Pp^\beta)$:
\begin{equation}\label{Tt,mut,Xt}
\hat{\mathbb T}_{\Pe,t}^ku(\mu)=\mathbf E_{x_\mu}\left[u_k\left(\frac1k\sum_{i=1}^k\delta_{X_t^{k,i}}\right)\right]
\end{equation}
for all Borel functions $u\in L^2(\Pe,\Pp^\beta)$ and a.e. $\mu\in\Pe$ and with $x_\mu:=\I_k(\mu)$.

\end{subsection}
\end{section}
\begin{section}{Convergence}

\begin{subsection}{Convergence of Finite Dimensional Distributions}
Note that $H(2^k)\subset H^(2^n)$ for $k,n\in\N$, $k\le n$, and thus $\D^{2^k}\subset\D^{2^n}$, $\mathbb V^{2^k}\subset\mathbb V^{2^n}$.
According to section 5.1
\begin{equation}
\mathbb T_t^{2^k}u\to \mathbb T_tu\quad\mbox{in }L^2(\G,\Q^\beta)\qquad\mbox{as }k\to\infty\end{equation}
for all $u\in \mathbb V^\infty:=\bigcup_{n\in\N}\mathbb V^{2^n}$. The latter is a dense subset in $L^2(\G,\Q^\beta)$.
The previous in particular implies
\begin{equation}
\langle u,\mathbb T_t^{2^k}v\rangle_{L^2(\G,\Q^\beta)}\to\langle u, \mathbb T_tv\rangle_{L^2(\G,\Q^\beta)}\qquad\mbox{as }k\to\infty\end{equation}
for all $u,v\in \mathbb V^\infty$
and thus
\begin{equation}
\mathbf E_{\Q^\beta_k}\left[u(g_0^{2^k})\cdot v(g_t^{2^k})\right]\to\mathbf E_{\Q}\left[u(g_0)\cdot v(g_t)\right]\qquad\mbox{as }k\to\infty\end{equation}
for all $u,v\in \mathcal C(\G)$.

\bigskip

The Markov property of the processes $(g_t)_{t\ge0}$ and $(g_t^{2^k})_{t\ge0}$ together with their invariance w.r.t. the measures $\Q^\beta$ and $\Q^\beta_{2^k}$ allows to iterate this argumentation which then yields
\begin{eqnarray*}
\lefteqn{
\mathbf E_{\Q^\beta_{2^k}}\left[u_1(g_{t_1}^{2^k})\cdot u_2(g_{t_2}^{2^k})\cdot \ldots\cdot u_N(g_{t_N}^{2^k})\right]}\\
&=&
\int_\G u_1\cdot \mathbb T^{2^k}_{t_1-t_0}\left(u_2\cdot  T^{2^k}_{t_2-t_1}\left(u_3\cdot \ldots\cdot  T^{2^k}_{t_N-t_{N-1}}u_N\right)\ldots\right)\, d\Q^\beta_{2^k}\\
&&\qquad\qquad\downarrow\\
&=&
\int_\G u_1\cdot \mathbb T_{t_1-t_0}\left(u_2\cdot  T_{t_2-t_1}\left(u_3\cdot \ldots\cdot  T_{t_N-t_{N-1}}u_N\right)\ldots\right)\, d\Q^\beta\\
&=&
\mathbf E_{\Q^\beta}\left[u_1(g_{t_1})\cdot u_2(g_{t_2})\cdot \ldots\cdot u_N(g_{t_N})\right]
\end{eqnarray*}
as $k\to\infty$
for all $N\in \N$, all $0\le t_1<\ldots<t_N$ and all $u_1,\ldots, u_N\in \mathcal C(\G)$.
Since functions $U\in\mathcal C(\G^N)$ can be approximated uniformly by linear combinations of functions of the form $U(g_1,g_2,\ldots,g_n)=\prod_{n=1}^Nu_n(g_n)$ it follows that
\begin{eqnarray*}
\mathbf E_{\Q^\beta_{2^k}}\left[U(g_{t_1}^{2^k}, g_{t_2}^{2^k}, \ldots,g_{t_N}^{2^k})\right]
&\to&
\mathbf E_{\Q^\beta}\left[U(g_{t_1}, g_{t_2}, \ldots,g_{t_N})\right]
\end{eqnarray*}
as $k\to\infty$
for all $N\in \N$, all $0\le t_1<\ldots<t_N$ and all $U\in \mathcal C(\G^N)$.
That is, we have proven the convergence
\begin{equation}\label{conv-fdd}
(g_t^{2^k})_{t\ge0}\ \to \ (g_t)_{t\ge0}\qquad\mbox{as }k\to\infty
\end{equation}
in the sense of weak convergence of the finite dimensional distributions of the processes, started with their respective invariant distributions.
By means of the various isomorphisms presented before, this can be equivalently restated as convergence
\begin{equation}
(\mu_t^{2^k})_{t\ge0}\ \to \ (\mu_t)_{t\ge0}\qquad\mbox{as }k\to\infty,
\end{equation}
again in the sense of weak convergence of the finite dimensional distributions of the processes, started with their respective invariant distributions.
Here $(\mu_t)_{t\ge0}$ denotes the Wasserstein diffusion on $\Pe$ -- associated with the Dirichlet form (\ref{Wasserstein-P}) -- with the entropic measure $\Pp^\beta$ as invariant distribution and
$$\mu_t^{2^k}(\omega)=\frac1{2^k}\sum_{i^=1}^{2^k}\delta_{X_t^{2^k,i}(\omega)}$$
with $\left(X_t^{2^k,i}\right)_{t\ge0}$ being the continuous Markov process on the simplex $\overline{\Sigma}_{2^k}$ -- associated with the Dirichlet form (\ref{Dir-endlich}) --  with invariant distribution $\rho^\beta_{2^k}(x)dx$.
\end{subsection}

\begin{subsection}{Convergence of Processes}
Convergence of the processes
\begin{equation*}
(g_t^{2^k})_{t\ge0}\ \to \ (g_t)_{t\ge0}\qquad\mbox{as }k\to\infty
\end{equation*}
will follow from the convergence (\ref{conv-fdd}) of the respective finite dimensional distributions  provided we prove tightness of the family $(g_t^{2^k})_{t\ge0}, k\in\N$ in $\C(\R_+,\G)$. The latter is equivalent to tightness of $\left(\langle\psi,g_t^{2^k}\rangle\right)_{t\ge0}, k\in\N$ in $\C(\R_+,\R)$ for all $\psi\in L^2([0,1],\leb)$. It suffices to verify this for a dense subset of $\psi$, e.g. for all $\psi\in\bigcup_{l=1}^\infty \HH(2^l)\subset L^2([0,1],\leb)$.

\smallskip

Fix $\psi\in\HH(2^l)$ for some $l\in\N$ with $\|\psi\|=1$. For each $k\in\N, k\ge l$ the continuous function $u(g):=\langle\psi,g\rangle$ lies in $\mathbb V_{2^k}$ with energy $\EE(u)=\|\psi\|^2=1$ and square field operator
\begin{equation}\label{Gamma}
\Gamma_{\langle u\rangle}(g)=1
\end{equation}
for a.e. $g\in\G$.

\smallskip

Given $T>0$, the process
$$\left(u(g_t^{2^k})\right)_{t\in[0,T]}$$
admits a Lyons-Zheng decomposition
$$u(g_t^{2^k})-u(g_0^{2^k})=\frac12 M_t^{(2^k)}-\frac12\left[ M_T^{(2^k)}- M_{T-t}^{(2^k)}\right]\circ r_T$$
into a forward martingale and a backward martingale. According to (\ref{Gamma}) the quadratic variation of the forward martingale -- as well as that of the backward martingale --  is given by
$$\langle M^{(2^k)}\rangle_t=t,$$
uniformly in $g\in\G$ and $k\in\N, k\ge l$. Hence, using hitting probabilities of 1-dimensional Brownian motions we deduce for any $R>0$ and uniformly in $k\in\N, k\ge l,$
\begin{eqnarray*}
\lefteqn{
\mathbf P_{\Q_{2^k}^\beta}\left[\sup_{t\in[0,T]}\left(u(g_t^{2^k})-u(g_0^{2^k})\right)>R\right]}\\
&\le&
\mathbf P_{\Q_{2^k}^\beta}\left[\sup_{t\in[0,T]}M_t^{(2^k)}>R\right]+
\mathbf P_{\Q_{2^k}^\beta}\left[\sup_{t\in[0,T]}\left(M^{(2^k)}_T-M^{(2^k)}_{T-t}\right)\circ r_T>R\right]\\
&\le&
2\sqrt{\frac2\pi} \exp\left(-\frac{(R/2)^2}{2T}\right) .
\end{eqnarray*}
Since we already know that the 1-dimensional distributions $g_0^{2^k}$ converge, this proves tightness of the family of processes
$$\left(u(g_t^{2^k})\right)_{t\in[0,T]}=\left(\langle\psi,g_t^{2^k}\rangle\right)_{t\in[0,T]}$$
for $k\in\N$.
Since this holds for all $\psi\in \bigcup_{l=1}^\infty \HH(2^l)$ it implies tightness of the family $(g_t^{2^k})_{t\ge0}, k\in\N$, and thus convergence
of the processes
\begin{equation*}
(g_t^{2^k})_{t\ge0}\ \to \ (g_t)_{t\ge0}\qquad\mbox{as }k\to\infty.
\end{equation*}

\smallskip

Applying the usual isomorphism, this may be restated as convergence of the processes
\begin{equation*}
(\mu_t^{2^k})_{t\ge0}\ \to \ (\mu_t)_{t\ge0}\qquad\mbox{as }k\to\infty
\end{equation*}
in $\C(\R_+\Pe)$.
\end{subsection}
\begin{subsection}{Final Remarks}
Given $k\in\N$ a mapping $\tilde \J_k:\G\to\Sigma_k$ -- very similar to our mapping $\hat \J_k$ from (4.1) -- is obtained by replacing the functions $\Phi_k^{(i)}$ from (4.2) by
$\tilde\Phi_k^{(i)}(x):=1_{[0,\frac{2i-1}{2k}]}(x)$ which leads to
$$\tilde\J_k(g)=\left(\int_0^1\tilde\Phi_k^{(i)}\,dg\right)_{i=1,\ldots,k}=\left(g\left(\frac{2i-1}{2k}\right)\right)_{i=1,\ldots,k}.$$
In this case, the identification of the push forward measure $\tilde m_k^\beta:=(\tilde\J_k)_*\Q^\beta$ on $\Sigma_k$ is much easier. Indeed,
it is absolutely continuous with density
$$\tilde \rho_k(x)=C\cdot
\left[x_1(1-x_k)\right]^{\beta/(2k)-1}\cdot\prod_{i=2}^{k}\left(x_i-x_{i-1}\right)^{\beta/k-1}.$$
The strong Markov process on $\overline\Sigma_k$ associated with the Dirichlet form
$\tilde{\mathcal E}_k(U)=k\int_{\Sigma_k}|\nabla U|^2\tilde \rho_k^\beta\,dx$ on $L^2(\Sigma_k,\tilde\rho_k^\beta\,dx)$ admits a very explicit characterization: at least on those stochastic intervals on which the process is in the interior of the simplex it is a weak solution to the
coupled system of stochastic differential equations
\begin{equation}
dX_t^{k,i}=
\left[\frac{\beta_{i-1}-k}{X_t^{k,i-1}-X_t^{k,i}}-\frac{\beta_i-k}{X_t^{k,i}-X_t^{k,i+1}}\right]dt +  \sqrt{2k}\, dW_t^i,\quad i=1,\ldots,k
\end{equation}
for some $k$-dimensional Brownian motion $(W_t)_{t\ge 0}$ and with $X_t^{k,0}:=0, X_t^{k,k+1}:=1$. Here
$\beta_0=\beta_k=\beta/2$ and $\beta_i=\beta$ for $i=1,\ldots,k-1$.
This is essentially the approximation used by S. Andres and M.-K- von Renesse \cite{AndRes}.

The fundamental disadvantage, however, is that the functions $g\mapsto \int_0^1\tilde\Phi_k^{(i)}\,dg$ are no longer in the domain of the Dirichlet form $\EE$. More generally, for any non-constant $U\in\C^1(\R^k)$ the function
$u(g):=U(\tilde\J_k(g))$ is neither continuous on $\G$ nor does it belong to $\D$.
\end{subsection}
\end{section}

\bibliographystyle{alpha}

\def\cprime{$'$} \def\cprime{$'$}

\end{document}